\documentclass[10pt,a4paper]{article}
\usepackage{amsmath,amsfonts,amssymb}
\usepackage{amsthm}
\usepackage{ifthen}
\usepackage{longtable}
\usepackage{array}
\usepackage{url}
\usepackage{multirow}

\newtheorem{definition}{Definition}
\newtheorem{theorem}{Theorem}
\newtheorem{lemma}{Lemma}

\newtheorem{hyp}{Hypothesis}

\theoremstyle{remark}
\newtheorem{remark}{Remark}

\makeatletter\def\blfootnote{\xdef\@thefnmark{}\@footnotetext}\makeatother

\begin{document}
\setlength{\parindent}{0pt}
\title{Ergodic properties of $\boldsymbol \beta$-adic Halton sequences}
\author{Markus Hofer\footnote{Johannes Kepler University,
Institute of Financial Mathematics, Altenbergerstrasse 69, 4040 Linz, Austria. \mbox{e-mail}: \texttt{markus.hofer@tugraz.at}. The author is supported by the Austrian Science Fund (FWF), Project P21943}
\quad Maria Rita Iac\`o \footnote{Graz University of Technology, Institute of Mathematics A, Steyrergasse 30, 8010 Graz, Austria and University of Calabria, Department of Mathematics and Computer Science, Via P. Bucci 30B, 87036 Arcavacata di Rende (CS), Italy. \mbox{e-mail}: \texttt{iaco@math.tugraz.at}. The author is partially supported by the Austrian Science Fund (FWF): W1230, Doctoral Program ``Discrete Mathematics''.}
\quad Robert Tichy \footnote{Graz University of Technology,
Institute of Mathematics A, Steyrergasse 30, 8010 Graz, Austria. \mbox{e-mail}: \texttt{tichy@tugraz.at}.}}
\date{}
\maketitle


\begin{abstract}
 We investigate a parametric extension of the classical $s$-dimensional Halton sequence, where the bases are special Pisot numbers. In a one-dimensional setting the properties of such sequences have already been investigated by several authors \cite{bg,carbone,ninomiya,steiner2}. We use methods from ergodic theory to in order to investigate the distribution behavior of multidimensional versions of such sequences. As a consequence it is shown that the Kakutani-Fibonacci transformation is uniquely ergodic.
\end{abstract}

\section{Introduction}

In this article we consider the distribution properties of deterministic point sequences in $[0,1)^s$. We use the following notation: for two points $\mathbf{a},\mathbf{b} \in [0,1)^s$  we write $\mathbf{a} \leq \mathbf{b}$ and $\mathbf{a} < \mathbf{b}$ if the corresponding inequalities hold in each coordinate; furthermore, we write $[\mathbf{a}, \mathbf{b})$ for the set $\{\mathbf{x} \in [0,1)^s: ~\mathbf{a} \leq \mathbf{x} < \mathbf{b}\}$, and we call such a set an $s$-dimensional interval. Moreover we denote by $\mathbf{1}_{I}$ the indicator function of the set $I \subseteq [0,1)^s$ and by $\lambda_s$ the $s$-dimensional Lebesgue measure, for short we write $\lambda$ instead of $\lambda_1$. Note that vectors will be written in bold fonts and we write $\mathbf{0}$ for the $s$-dimensional vector $(0,\dots,0)$.\\

A sequence $(x_n)_{n \in \mathbb{N}}$ of points in $[0,1)^s$ is called uniformly distributed modulo 1 (u.d.) if
\begin{equation*}
 \lim_{N \rightarrow \infty} \frac{\sum_{n = 1}^N \mathbf{1}_{[\mathbf{a}, \mathbf{b})} (x_n)}{N} = \lambda_s([\mathbf{a},\mathbf{b}))
\end{equation*}
for all $s$-dimensional intervals $[\mathbf{a}, \mathbf{b} ) \subseteq [0,1)^s$. A further characterization of uniform distribution is due to Weyl~\cite{weyl}: a sequence $(x_n)_{n \in \mathbb{N}}$ of points in $[0,1)^s$ is u.d.\ if and only if for every continuous function $f$ on $[0,1)^s$ the relation
\begin{equation*}
  \lim_{N \rightarrow \infty} \frac{\sum_{n = 1}^N f(\mathbf{x}_n)}{N} = \int_{[0,1)^s} f(\mathbf{x}) d\mathbf{x}
\end{equation*}
holds. Weyl's criterion suggests a numerical integration technique which is usually called Quasi Monte Carlo (QMC) integration. However, the theorem gives no information on the quality of the estimator.\\

The Koksma--Hlawka inequality~\cite{hlawka} states that the integration error of QMC integration can be bounded by the product of the variation of $f$ (in the sense of Hardy and Krause), denoted by $V(f)$, and the so-called star-discrepancy $D_N^*$ of the point sequence $(\mathbf{x}_n)_{n \in \mathbb{N}}$, i.e.\
\begin{equation*}
 \left| \frac{1}{N} \sum_{n = 1}^N f(\mathbf{x}_n) - \int_{[0,1]^s} f(\mathbf{x}) d\mathbf{x} \right| \leq V(f) D_N^*(\mathbf{x}_n),
\end{equation*}
where $D_N^*$ is defined by
\begin{equation*}
 D_N^* = D_N^*(\mathbf{x}_1, \ldots, \mathbf{x}_N) = \sup_{\mathbf{a} \in [0,1)^s} \left| \frac{\sum_{n=1}^N \mathbf{1}_{(\mathbf{0},\mathbf{a})} (\mathbf{x}_n)}{N} - \lambda_s([\mathbf{0},\mathbf{a}))\right|.
\end{equation*}
In order to minimize the integration error we have to use point sequences with small discrepancy. There are several constructions for sequences which achieve a star-discrepancy of order $\mathcal{O}(N^{-1} (\log N)^s)$, so-called low-discrepancy sequences. Note that this convergence rate, which is best possible among known sequences, is for all $s \geq 1$ better than that of the probabilistic error of the standard Monte Carlo method, where a sequence of random instead of deterministic points is used. QMC integration is successfully applied in several different areas of applied mathematics, for example in actuarial or financial mathematics, where high-dimensional numerical integrals appear frequently, see e.g.\ \cite{asal, okt}.\\ 

In this article we will construct point sequences by a combination of methods from uniform distribution theory and dynamical systems.
\begin{definition}
Let $(X,\mathcal{A},\mu)$ be a probability space. A  measurable transformation $T:X\rightarrow X$ is called ergodic if for every $A\in\mathcal{A}$ such that 
$T^{-1}(A) = A$, either $\mu(A)=0$ or $ \mu(A)=1$.
\end{definition}

The system $(X,\mathcal{A},\mu, T)$ is called a measure theoretical dynamical system, or dynamical system, for short.
 If $T$ is ergodic, the system is called ergodic.\\

The link between dynamical systems and uniform distribution is given by the following classical result of Birkhoff.

\begin{theorem}[Pointwise Ergodic Theorem]\label{B}
Let $(X,\mathcal{A},\mu, T)$ be a dynamical system. Then, for every $f \in \mathcal{L}^1(X)$ $$\lim_{N\to\infty}\frac{1}{N}\sum_{j=0}^{N-1}f(T^{j}x)$$ 
exists for $\mu$-almost every $x\in X$ (here $T^0x=x$). Furthermore if $T:X\rightarrow X$ is ergodic, then for every $f\in \mathcal{L}^{1}(X)$ we have $$\lim_{N\to\infty}\frac{1}{N}\sum_{j=0}^{N-1}f(T^{j}x)=\int_Xf(x)d\mu(x)\ ,$$ for $\mu$-almost every $x\in X$.
\end{theorem}

Thus if $(\mathbf{x}_n)_{n \in \mathbb{N}}$ is constructed as the orbit of a point $\mathbf{x} \in [0,1)^s$ with respect to an ergodic transformation $T$ on $[0,1)^s$, i.e.\ $(\mathbf{x}_n)_{n \in \mathbb{N}} = (T^n \mathbf{x})_{n \in \mathbb{N}}$, we can easily see the connection to Weyl's criterion. Nevertheless, the statement of Theorem \ref{B} is only valid $\mu$-almost sure, so we need stronger conditions in order to assure that the orbit of a certain (or every) point $\mathbf{x} \in [0,1)^s$ under the transformation $T$ is u.d.

\begin{definition}
 A continuous transformation $T:X\longrightarrow X$ on a compact metrizable space $X$ is called uniquely ergodic if there is only one $T$-invariant Borel probability measure on $X$. The system $(X,\mathcal{A},\mu, T)$ is called uniquely ergodic.
\end{definition}

If the transformation $T$ is uniquely ergodic, then Theorem \ref{B} holds for every $x\in X$ and so the associated sequence $(T^{n}x)_{n\in \mathbb{N}}$ is u.d., see e.g.\ \cite[Theorem 6.19]{Walters}.\\ 

A way to further analyze the properties of a dynamical system is to find an isomorphism with a well-known ergodic or uniquely ergodic system.
\begin{definition}
Two dynamical systems $(X_i,\mathcal{A}_i,\mu_i,T_i)$ $(i=1,2)$ are called isomorphic if there exist two sets $A_1\in\mathcal{A}_1$ and $A_2\in\mathcal{A}_2$, with $\mu_i(A_i)=1$ 
$(i=1,2)$ and a bijection $\varphi:A_1\longrightarrow A_2$ such that
\begin{equation*}
 T_2\circ \varphi(x)= \varphi\circ T_1(x)\ , \qquad \forall x \in A_1\ .
\end{equation*}
\end{definition}

A classical example of a uniquely ergodic system is $(\mathbb{Z}_b,\tau)$ (see for instance, \cite{Walters}), where $\mathbb{Z}_b$ is the compact group of $b$-adic integers 
and $\tau:\mathbb{Z}_b\longrightarrow \mathbb{Z}_b$ the addition-by-one map (called odometer). We will shortly recall the connection between $(\mathbb{Z}_b,\tau)$ and low-discrepancy sequences on $[0,1)$. For an integer $b\geq 2$, every $z \in \mathbb{Z}_b$ has a unique expansion of the form
\begin{equation*}
 z=\sum_{j\geq 0}z_jb^j
\end{equation*}
with digits $z_j\in \{0,1,\dots, b-1\}$. For $z \in \mathbb{Z}_b$ we define the $b$-adic Monna map $\varphi_b:\mathbb{Z}_b\longrightarrow [0,1)$, see also \cite{ghl}, by
\begin{equation*}
 \varphi_b\left(\sum_{j\geq 0}z_jb^j  \right)= \sum_{j\geq 0}z_jb^{-j-1}.
\end{equation*}
The restriction of $\varphi_b$ to $\mathbb{N}_0$ is called radical-inverse function in base $b$ and the sequence 
\begin{equation*}
 (\varphi_b(n))_{n\in \mathbb{N}}
\end{equation*}
is the so-called van der Corput sequence in base $b$ which is a low-discrepancy sequence in $[0,1)$.\\

The Monna map is continuous and surjective but not injective. In order to make it an isomorphism we only consider the so-called regular representations, i.e. representations with infinitely 
many digits $z_j$ different from $b-1$. So the Monna map restricted to these regular representations admits an inverse (called pseudo-inverse) $\varphi_b^+:[0,1)\longrightarrow \mathbb{Z}_b$, defined by
\begin{equation*}
 \varphi_b^+\left(\sum_{j\geq 0}z_jb^{-j-1}  \right)= \sum_{j\geq 0}z_jb^j\ ,
\end{equation*}
where $\sum_{j\geq 0}z_jb^{-j-1}$ is a $b$-adic rational in $[0,1)$. \\

Moreover $\varphi_b$ is measure preserving from $\mathbb{Z}_b$ onto $[0,1)$ and it transports the normalized Haar measure on $\mathbb{Z}_b$ to the Lebesgue measure on $[0,1)$. Hence by the unique 
ergodicity of $\tau$ it follows that the sequence $(\tau^n z)_{n\geq 0}$ is uniformly distributed in $\mathbb{Z}_b$ for all $z\in\mathbb{Z}_b$, in particular for $z=0$. Thus the van der Corput sequence $(\varphi_b(\tau^n0))_{n\in \mathbb{N}}$ in base $b$ is uniformly distributed modulo 1.\\

In order to construct multidimensional sequences we need a criterion to ensure that the Cartesian product of several ergodic systems is again ergodic, see e.g.\ \cite{ghl}.
\begin{theorem}\label{product}
 Let $\mathcal{T}_i=(X_i,\mathcal{A}_i,\mu_i,T_i), ~i=1,\ldots,s,$ be uniquely ergodic dynamical systems. Then the dynamical system $\mathcal{T}_1 \times \ldots \times \mathcal{T}_s$ is uniquely ergodic if and only if for all $i,j \in \{1, \ldots, s\}, ~i \neq j,$ the discrete parts of the spectra of $T_i$ and $T_j$ intersect only at 1.
\end{theorem}
In \cite{ghl}, the spectrum of $(\mathbb{Z}_{b}, \tau)$ is given explicitly.  Furthermore the authors show that if $b_1, b_2$ denote positive integers and $\tau_{b_i}$ denotes the addition-by-one on $\mathbb{Z}_{b_i}$, then the dynamical systems $\mathbb{Z}_{b_1}$ and $\mathbb{Z}_{b_2}$ are spectrally disjoint if and only if $b_1$ and $b_2$ are coprime. This is exactly the condition proved by Halton in \cite{Halton} in order to obtain a low-discrepancy sequence in $[0,1)^s$ by combining coordinatewise van der Corput sequences. The resulting sequence $(\phi_\mathbf{b}(n))_{n \in \mathbb{N}} = (\phi_{b_1}(n), \ldots, \phi_{b_d}(n))_{n \in \mathbb{N}}$ is called $\mathbf{b}$-adic Halton sequence, where $\mathbf{b}$ is an $s$-dimensional vector of pairwise coprime integers $b_i, ~i=1, \ldots, d$.\\

The aim of the present article is to extend the above idea to point sequences with irrational bases. Such sequences in the unit interval were investigated by several authors. Barat and Grabner \cite{bg} consider the so-called $\beta$-adic van der Corput sequence $(\phi_\beta(n))_{n \in \mathbb{N}}$ on $[0,1)$ and similar constructions. They prove that $(\phi_\beta(n))_{n \in \mathbb{N}}$ is low-discrepancy, where $\beta$ is the characteristic root of special linear recurrences. Ninomiya \cite{ninomiya} considers the discrepancy of point sequences on $[0,1)$ for a slightly greater class of irrational bases $\beta$. In \cite{im} the underlying construction was extended to piecewise linear maps. Furthermore Steiner \cite{steiner1} considers so called bounded remainder sets in this setting. In a second article Steiner \cite{steiner2} considers van der Corput sequences on abstract numeration systems and gives conditions under which they are low discrepancy. Discrepancy bounds for a higher dimensional extension of \cite{ninomiya} are given in \cite{mori}. Note that the construction in \cite{mori} is different from that in the present article.\\

Carbone \cite{carbone} and Drmota and Infusino \cite{drmota} investigate the discrepancy of point sequences generated by the so-called Kakutani splitting procedure. Carbone completely characterises the growth order of the discrepancy for a two parametric subfamily, the so-called LS-sequences. Moreover, Aistleitner, Hofer and Ziegler \cite{ahz} give conditions under which an $s$-dimensional vector of LS-sequences is not u.d.\ in $[0,1)^s$.\\ 

The remainder of the article is structured as follows: in the next section we formulate a characterization of uniquely ergodic systems which are constructed as Cartesian products of odometers on numeration systems related to linear recurrences. In the third section we give conditions under which a construction like that of Halton produces an u.d.\ sequence on $[0,1)^s$ with respect to irrational bases $\beta_1, \ldots, \beta_s$. Furthermore we present a parametric class of sequences which satisfies these conditions. Finally we prove that the ergodic Kakutani-Fibonacci-transformation, presented in \cite{carbone3}, is in fact uniquely ergodic.

\section{General $G$-odometers}

In this section we consider odometers on numeration systems, which are related to linear recurrences. For a detailed discussion of such number systems we refer to \cite{Fraenkel, glt, gt, gt2}. We first recall some basic results and definitions and then investigate Cartesian products of odometers.

\begin{definition}
 Let $(G_n)_{n \geq 0}$ be an increasing sequence of positive integers with $G_0=1$. Then every positive integer can be expanded in the following way
\begin{equation}\label{greedy}
 \forall n\in\mathbb{N}\ ,\quad n=\sum_{k=0}^{\infty}\varepsilon_k G_k,
\end{equation}
where $\varepsilon_k \in \{0, \ldots, \lfloor G_{k+1}/G_k \rfloor \}$ and $\lfloor x \rfloor$ denotes the integral part of $x$, that is the greatest integer less than or equal to $x \in \mathbb{R}$. This expansion (called $G$-expansion) is uniquely determined and finite, provided that for every $K$
\begin{equation}\label{eq1}
\sum_{k=0}^{K-1}\varepsilon_k G_k < G_K .
\end{equation}
\end{definition}
The digits $\varepsilon_k$ can be computed by the greedy algorithm (see for instance \cite{Fraenkel}) and the sequence $G = (G_n)_{n \geq 0}$ is called numeration system.\\

We denote by $\mathcal{K}_G$ the subset of sequences that verify the property (\ref{eq1}) and the elements in $\mathcal{K}_G$ are called G-admissible. In order to extend the addition-by-one map defined on $\mathbb{N}$ to $\mathcal{K}_G$ we introduce $\mathcal{K}_G^0 \subseteq \mathcal{K}_G$
\begin{equation}
 \mathcal{K}_G^0=\left\lbrace x\in \mathcal{K}_G\ : \exists M_x, \forall j\geq M_x\quad \sum_{k=0}^{j}\varepsilon_k G_k\ < G_{j+1}-1\right\rbrace\ .
\end{equation}
If we denote by $x(j)=\sum_{k=0}^{j}\varepsilon_k G_k$, then we set
\begin{equation}\label{eq2}
 \tau(x)=(\varepsilon_0(x(j)+1)\dots \varepsilon_j(x(j)+1))\varepsilon_{j+1}\varepsilon_{j+2}\dots \ ,
\end{equation}
for every $x\in \mathcal{K}_G^0$ and $j\geq M_x$. This definition does not depend on the choice of $j\geq M_x$ and can be easily extended to sequences $x$ in $\mathcal{K}_G\setminus \mathcal{K}_G^0$ by $\tau(x)=0=(0^{\infty})$. In this way the transformation $\tau$ is defined on $\mathcal{K}_G$ and it is called $G$-odometer. We refer to \cite{glt} for a complete survey on odometers related to general numeration systems.\\

In this article we consider only numeration systems where the base sequence is a linear recurrence. Let $G_0=1$ and $G_k=a_0G_{k-1}+\dots +a_{k-1}G_0+1$ for $k<d$. Then $G_n$ for $n \geq d$ 
is determined by a recurrence of order $d \geq 1$, i.e.\
\begin{equation}\label{rec}
 G_{n+d}=a_0G_{n+d-1}+\dots + a_{d-1}G_n \qquad n\geq 0\ .
\end{equation}
The solution of the characteristic equation of the numeration system $G$
\begin{equation}\label{alpha}
 x^{d} = a_0 x^{d - 1} + \ldots + a_{d - 1}.
\end{equation}
plays a central role. We will be mainly interested in numeration systems where the solution of \eqref{alpha} is a Pisot number $\beta$. Note that $\beta$ is always a Pisot number if
\begin{equation}\label{descent}
a_0 \geq \ldots \geq a_{d-1} \geq 1,
\end{equation} 
see \cite[Theorem 2]{brauer}. By \cite{parry} we get that in this case the so-called Parry's $\beta$-expansion of $\beta$ is finite, i.e.\
\begin{equation}\label{par}
 \beta = a_0 + \frac{a_1}{\beta} + \ldots + \frac{a_{d - 1}}{\beta^{d - 1}},
\end{equation}
where $a_0 = \lfloor \beta \rfloor$. At the end of the last section we will also consider numeration systems where \eqref{descent} does not hold.\\

For numeration systems where the characteristic root $\beta$ is a Pisot number which satisfies \eqref{par}, we have that a finite sum $\sum_{k=0}^{\infty}\varepsilon_kG_k$ is the expansion of some integer if and only if the digits 
$\varepsilon_k$ of the $G$-expansion satisfy
\begin{equation}\label{maximal}
 (\varepsilon_k,\varepsilon_{k-1},\dots,\varepsilon_0,0^{\infty})< (a_0,a_1,\dots,a_{d-1})^{\infty}\ ,
\end{equation}
for every $k$ and $<$ denoting the lexicographic order (see \cite{parry}). Representations $(\varepsilon_k,\dots,\varepsilon_0)$ verifying this condition are called admissible representations and so they belong to $\mathcal{K}_G$.\\

In \cite[Theorem 5]{glt}, the authors show that the odometer on an admissible numeration system $G$ is uniquely ergodic and that the corresponding unique invariant measure $\mu$ is given by
\begin{align}
 &\mu(Z) \label{mu} =\\ 
 &\frac{F_{K} \beta^{d-1} + (F_{K + 1} - a_0 F_K) \beta^{d-2} + \ldots + (F_{K+d-1} - a_0 F_{K+d-2} - \ldots - a_{d-2} F_K)}{\beta^K (\beta^{d-1} + \beta^{d-2} + \ldots + 1)},\notag
\end{align}
where $F_K := \# \{ n < G_K: n \in Z \}$ and $Z$ is the cylinder with fixed digits $\epsilon_0, \ldots, \epsilon_{K-1}$. Note that the formula in \cite[Theorem 5]{glt} included a misprint and was stated in corrected form in \cite{bg}.\\

In the sequel we want to apply Theorem \ref{product}, thus we need information on the spectrum of the $G$-odometer. We introduce the following two hypotheses:
\begin{hyp}[Grabner, Tichy and Liardet \cite{glt}]\label{hypA}
There exists an integer $b > 0$ such that for all $k$ and 
\begin{equation*}
 N = \sum_{i = 0}^k \epsilon_i G_i + \sum_{j = k + b + 2}^\infty \epsilon_j G_j,
\end{equation*}
the addition of $G_m$ to $N$, where $m \geq k + b + 2$, does not change the digits $\epsilon_0, \ldots, \epsilon_k,$ in the greedy representation i.e.\
\begin{equation*}
N + G_m = \sum_{i = 0}^k \epsilon_i G_i + \sum_{j = k + 1}^\infty \epsilon'_j G_j.
\end{equation*}
\end{hyp}

\begin{hyp}[Frougny and Solomyak \cite{frougny}]\label{hypB}
 The solution $\beta$ of equation \eqref{alpha} is a Pisot number such that all numbers of the set $\mathbb{Z}[\beta^{-1}]$ have finite $\beta$-expansions.
\end{hyp}

In \cite{glt} the authors remark that the Multinacci sequence, i.e.\ $a_0 = \ldots = a_{d-1} = 1$, fulfills Hypothesis \ref{hypA}. Several authors worked on algebraic characterizations of Pisot numbers $\beta$ which satisfy Hypothesis \ref{hypB}. Frougny and Solomyak show that \eqref{alpha} implies Hypothesis \ref{hypB} and they give a full characterization of all Pisot numbers of degree two with this property. Furthermore Hollander \cite{holl} states another sufficient condition for Hypothesis \ref{hypB} and Akiyama \cite{aki} characterizes all Pisot units of degree three satisfying Hypothesis \ref{hypB}. Further progress was also made by Akiyama et al.\ \cite{aki2} who prove Hypothesis \ref{hypB} for a large class of Pisot numbers of degree three by using the theory of shift radix systems. Nevertheless there exists no complete algebraic characterization for Pisot numbers satisfying Hypothesis \ref{hypB} of degree greater than two. Note that both hypotheses can be satisfied by the same numeration system but, to the best of the authors knowledge, it is unknown if the two hypotheses are equivalent, see \cite{glt}.\\

Grabner, Liardet and Tichy \cite[Theorem 6]{glt} and Solomyak \cite[Theorem 4.1]{solomyak} show that the odometer on the base system $G$ has purely discrete spectrum provided that one of the above Hypotheses holds. Furthermore we obtain in both cases that the set of eigenvalues of the transformation is given by 
\begin{equation}\label{Gamma}
 \Gamma := \{z \in \mathbb{C}: \lim_{n \rightarrow \infty} z^{G_n} = 1 \}.
\end{equation}

\begin{theorem}\label{main1}
 Let $G^1, \ldots, G^s$ be numeration systems given by \eqref{rec}. Assume that the coefficients of the linear recurrences are given as $a_j^i = b_i, ~i = 1, \ldots, s, ~j = 0, \ldots, (d_i - 1),$ with pairwise coprime, positive integers $b_i, i = 1, \ldots, s$. Furthermore let $\frac{\beta_i^{k}}{\beta_j^l} \notin \mathbb{Q}$, for all $l, k \in \mathbb{N}$, where $\beta_1, \ldots, \beta_s$ are the roots of the characteristic equations \eqref{alpha}. Then the dynamical system which is constructed as the $s$-dimensional Cartesian product of the corresponding odometers, i.e.\ $((\mathcal{K}_{G^1}, \tau_1) \times \ldots \times (\mathcal{K}_{G^s}, \tau_s))$, is uniquely ergodic.
\end{theorem}
\begin{proof}
 It follows by \cite[Main Theorem]{solomyak} that the $G^i$ fulfill Hypothesis \ref{hypB} and thus the components of the $s$-dimensional dynamical system are uniquely ergodic. Furthermore we obtain that their spectrum is given by \eqref{Gamma}. By Theorem \ref{product}, we derive that the Cartesian product is uniquely ergodic if and only if $\Gamma_i \cap \Gamma_j = {1}$ for all $1 \leq i < j \leq d$. As noted in \cite{glt}, we have the following connection between $\beta_i$ and the corresponding sequence $G^i_n$,
\begin{equation}\label{approx}
 \lim_{n \rightarrow \infty} \frac{G^i_n}{\beta_i^n} = C_i,
\end{equation}
where the constant $C_i$ can be computed by residue calculus. Now consider a fixed $l \in \mathbb{N}$ and
\begin{align*}
 \exp \left( 2 \pi i \frac{G^i_n}{\beta_i^l} \right) &\approx \exp \left( 2 \pi i C_i \beta_i^{n - l} \right)\\
 &\approx \exp \left( 2 \pi i G^i_{n - l} \right),
\end{align*}
and thus
\begin{equation*}
 \lim_{n \rightarrow \infty} \exp \left( 2 \pi i \frac{G^i_n}{\beta_i^l} \right) = \lim_{n \rightarrow \infty} \exp \left( 2 \pi i G^i_{n - l} \right) = 1,
\end{equation*}
where $C_i$ is given in \eqref{approx}. Furthermore, it is easy to see that for every $k \in \mathbb{N}$ there exists a $n_0$ with $b_i^k \mid G_{n}$ for all $n \geq n_0$ and there exist no $b', n_0' \in \mathbb{N}$ with $\gcd(b', b_i) = 1$ such that $b' \mid G_{n}$ for all $n \geq n_0'$. By simple considerations we get that $\Gamma_i$ can be written as
\begin{equation*}
 \Gamma_i = \left\{ \exp \left( 2 \pi i \frac{c_i}{b_i^m \beta_i^l} \right) \colon m,l,c_i \in \mathbb{N} \cup \{0\} \right\}.
\end{equation*}
An application of Theorem \ref{product} completes the proof.
\end{proof}

\section{Uniform distribution of the $\beta$-adic Halton sequence}

The purpose of this section is to formulate conditions on the odometers $(\mathcal{K}_{G^1}, \tau_1),$ $\ldots, (\mathcal{K}_{G^s}, \tau_s)$ such that their product dynamical system is uniquely ergodic and the Monna map transports the measure $\mu_1 \times \ldots \times \mu_s$, where the $\mu_i$ are given by \eqref{mu}, to the Lebesgue measure on $[0,1)^s$. Under such assumptions we show the resulting $s$-dimensional, $\beta$-adic Halton sequence to be u.d.\ in $[0,1)^s$.\\

First we extend the definition of the Monna map to irrational bases $\beta > 1$. Let
\begin{equation*}
 n = \sum_{j \geq 0} \epsilon_j G_j
\end{equation*}
be the $G$-expansion of an integer $n$. We define the $\beta$-adic Monna map $\phi_\beta \colon \mathcal{K}_G \rightarrow \mathbb{R}^+$ as
\begin{equation*}
\phi_\beta(n) = \phi_\beta \left(\sum_{j \geq 0} \epsilon_j G_j \right) = \sum_{j \geq 0} \epsilon_j \beta^{-j-1}\ .
\end{equation*}
We call
\begin{equation}\label{alphaexp}
 x = \sum_{j \geq 0} \epsilon_j \beta^{-j-1},
\end{equation}
the $\beta$-expansion of $x$. Furthermore, as in the first section, we define the radical inverse function as restriction of $\phi_\beta$ on $\mathcal{K}^0_G$ and define the pseudo-inverse $\phi_{\beta}^+$ similarly. In this context we define the $\boldsymbol \beta$-adic Halton sequence as $\phi_{\boldsymbol \beta}(n))_{n \in \mathbb{N}} = (\phi_{\beta_1}(n), \ldots, \phi_{\beta_s}(n))_{n \in \mathbb{N}}$, where $\boldsymbol \beta = (\beta_1, \ldots, \beta_s)$ and the $\beta_i$ are solution of the corresponding characteristic equations.\\ 

Note that even if a Pisot number $\beta$ is chosen as the solution of \eqref{alpha}, it is not sure that the image of $\mathcal{K}_0$ under $\phi_\beta$ is a subset of $[0,1)$ or dense in $[0,1)$. The following lemma gives a characterization of numeration systems for which this is true.

\begin{lemma}\label{vdC}
 Let $\mathbf{a} = (a_0, \ldots, a_{d-1})$, let the integers $a_0, \ldots, a_{d-1} \geq 0$ be the coefficients defining the numeration system $G$ and assume that the corresponding characteristic root $\beta$ satisfies \eqref{alpha}. Then $\phi_\beta(\mathbb{N}) \subset [0,1)$ and $\phi_\beta(\mathbb{N}) \not\subset [0,x)$ for all $0 < x < 1$ if and only if $\mathbf{a}$ can be written either as
\begin{align}
 \mathbf{a} &= (a_0, \ldots, a_0),\label{case1}\\
 \mathbf{a} &= (a_0, a_0 - 1, \ldots, a_0 - 1, a_0),\label{case2}\\
 \mathbf{a} &= (a_0, \ldots, a_0, a_0 + 1)\label{case4}
\end{align}
or
\begin{equation}\label{case3}
 \mathbf{a} = (\mathbf{a}', \ldots, \mathbf{a}', \mathbf{a}''),
\end{equation}
where $a_0 > 0$, $\mathbf{a}', \mathbf{a}''$ are of equal length and are of the form
\begin{align*}
 \mathbf{a}' &= (a_0, \ldots, a_0, a_0 - 1),\quad \mathbf{a}'' = (a_0, \ldots, a_0)\text{ or }\\
 \mathbf{a}' &= (a_0, a_0 - 1, \ldots, a_0 - 1),\quad \mathbf{a}'' = (a_0, a_0 - 1, \ldots, a_0 - 1, a_0) 
\end{align*}
\end{lemma}
\begin{proof}
It follows by \eqref{alpha} that
\begin{equation}\label{eins}
 \frac{a_0}{\beta} + \ldots + \frac{a_{d-1}}{\beta^{d}} = 1.
\end{equation}
Furthermore we know that for all admissible representations of an integer $n$ we have
\begin{equation*}
 (\varepsilon_k,\varepsilon_{k-1},\dots,\varepsilon_0,0^{\infty})< (a_0,a_1,\dots,a_{d-1})^{\infty}\ ,
\end{equation*}
for every $k$ and $<$ denoting the lexicographic order. If we have $a_0 \neq \max_{0 \leq i < d}(a_i)$, by \eqref{eins} we obtain
\begin{equation*}
 \max_{n \in \mathbb{N}} \phi_\beta(n) \leq \sum_{i = 1}^\infty \frac{a_0}{\beta^i} \leq 1,
\end{equation*}
where equality holds for both inequations only if $\mathbf{a}$ is of the form \eqref{case4}.\\

Assume now that $a_0 = \max_{0 \leq i < d}(a_i)$ and there exist an $0 < j < d$ such that $a_j < a_0 - 1$ and $k \geq 0$ is the maximal integer with $a_i = a_0$ for all $i \leq k$. Then the representation $0 < k < d-1$ and $(a_0, \ldots, a_k, (a_0-1, a_0, \ldots, a_{k - 1})^\infty)$ is admissible. But by \eqref{eins} we get that the image of the $\beta$-adic Monna map of this representation is strictly greater than 1. By a similar argument can be applied in the case $a_{d-1} \neq a_0$.\\

Note that we have excluded every case which is not of the form $(a_0, a_1, \ldots, a_{d-2}, a_0)$, where $a_i \in \{a_0 - 1, a_0\}$ for $1 \leq i \leq d-2$. By \eqref{eins} and the construction of the van der Corput sequence we obtain that the image of the $\beta$-adic Monna map of $(a_0, \ldots, a_{d-2}, a_{d-1} - 1)^\infty$ is 1. Thus it is sufficient to show that $(a_0, \ldots, a_{d-2}, a_{d-1} - 1)^\infty$ is admissible and its image under the Monna map is maximal. This is clear if $a$ is of the form \eqref{case1}-\eqref{case3}.\\ 

Finally assume that $\mathbf{a} = (a_0, a_1, \ldots, a_{d-2}, a_0)$, where $a_i \in \{a_0 - 1, a_0\}$ for $1 \leq i \leq d-2$, $\mathbf{a}$ is not included in one of these cases and let $k$ be defined as above. Then $0 < k < d-1$ and $(a_0, \ldots, a_k, (a_0-1, a_0, \ldots, a_{k - 1})^\infty)$ is maximal and admissible. But this is only equal to $(a_0, \ldots, a_{d-2}, a_{d-1} - 1)^\infty$ when \eqref{case3} holds.  
\end{proof}

\begin{remark}
 Note that \eqref{case3} is another way to represent number systems defined by $\mathbf{a}''$ satisfying \eqref{case1} or \eqref{case2}. For example the numeration system defined by $\mathbf{a}^* = (1,0,1,1)$ is the same as the Fibonacci numeration system where $\mathbf{a} = (1,1)$. Furthermore if $\mathbf{a}$ is of the form \eqref{case4} we can rewrite it as the classical $a_0$-adic numeration system. Hence in the sequel we will only be interested in numeration systems which fulfill \eqref{case1} or \eqref{case2}. 
\end{remark}

\begin{lemma}\label{lem3}
 Let $G$ be a numeration system of the form \eqref{rec}, assume that the coefficients of the linear recurrence are given by $a_j = a, ~j = 0, \ldots, (d-1),$ for a positive integer $a$ and let $\beta$ denote the corresponding characteristic root. Then $\mu(Z) = \lambda(\phi_\beta(Z))$ for every cylinder set $Z$.
\end{lemma}
\begin{proof}
 Let the cylinder set $Z$ be defined by the fixed digits $\epsilon_0, \ldots, \epsilon_{k-1}$. Assume first that $\epsilon_{k-1} < a$, then $F_{k+r} = (a+1)^r$ for $0 \leq r < d$. Thus, by \eqref{mu}, we obtain that
\begin{equation*}
 \mu(Z) = \beta^{-k}.
\end{equation*}

Consider the $\beta$-adic Monna map at $n \in \mathbb{N}$, i.e.\
\begin{equation*}
 \phi_\beta(n) = \sum_{i = 0}^\infty \frac{\epsilon_i}{\beta^{i+1}}.
\end{equation*}
If $\epsilon_{k-1} < a$ we can easily see that $\phi_\beta(Z)$ is dense in 
\begin{equation*}
I = \left[\sum_{i = 0}^{k-1} \frac{\epsilon_i}{\beta^{i+1}}, \sum_{i = 0}^{k-2} \frac{\epsilon_i}{\beta^{i+1}} + \frac{(\epsilon_{k-1} + 1)}{\beta^{k}} \right)
\end{equation*}
and that $\phi_\beta(x') \notin I$ if $x' \notin Z$. Thus $\phi_\beta(Z)$ is $\lambda$-measurable and $\lambda(\phi_\beta(Z)) = \lambda(I) = \beta^{-k}$.\\

Assume now that $Z$ is defined by the fixed digits $\epsilon_{0},\ldots , \epsilon_{k-2}$ and $\epsilon_{k-1} = a$. By the above argument we derive that a cylinder with fixed digits $\epsilon_{0} ,\ldots, \epsilon_{k-2}$ has measure $\beta^{-(k-1)}$ and every cylinder with digits $\epsilon_{0} , \ldots , \epsilon_{k-1}$ has measure $\beta^{-k}$. Thus
\begin{equation*}
 \mu(Z) = \beta^{-(k-1)} - (a - 1) \beta^{-k}.
\end{equation*}
Now we consider $\phi_\beta(Z)$, hence
\begin{equation*}
\phi_\beta(Z) = \left[\sum_{i = 0}^{k-1} \frac{\epsilon_i}{\beta^{i+1}}, \sum_{i = 0}^{k-3} \frac{\epsilon_i}{\beta^{i+1}} + \frac{(\epsilon_{k-2} + 1)}{\beta^{k}} \right)
\end{equation*}
and thus $\lambda(\phi_\beta(Z)) = \mu(Z)$.\\ 

Let $2 \leq h \leq \min(k, d-1)$ and consider a cylinder set $Z$ with fixed digits $\epsilon_{0},\ldots , \epsilon_{k-h-1} < a$ and $\epsilon_{k-l}=a$ for $l = 1, \ldots, h$. Then, as above we get that the cylinder with fixed digits $\epsilon_{0},\ldots , \epsilon_{k-h-1}$ has measure $\beta^{-(k-h)}$ and every cylinder with digits $\epsilon_{0},\ldots , \epsilon_{k-h+1}$ has measure $\beta^{-(k-h+2)}$. Thus we get that
\begin{equation*}
 \mu(Z) = \beta^{-(k-h+1)} - (a - 1) \beta^{-(k-h+2)}
\end{equation*}
Considering $\phi_\beta(Z)$ we have
\begin{equation*}
\phi_\beta(Z) = \left[\sum_{i = 0}^{k-h+1} \frac{\epsilon_i}{\beta^{i+1}}, \sum_{i = 0}^{k-h} \frac{\epsilon_i}{\beta^{i+1}} + \frac{(\epsilon_{k-h-1} + 1)}{\beta^{k-h+2}} \right),
\end{equation*}
and thus $\lambda(\phi_\beta(Z)) = \mu(Z)$.
\end{proof}

\begin{remark}
 As mentioned in the previous section, a result of Frougny and Solomyak \cite[Lemma 3]{frougny} implies that the dominant root of
\begin{equation*}
 x^2 - a_0 x - a_1, \quad a_0,a_1 \geq 1,
\end{equation*}
is a Pisot number if and only if $a_0 \geq a_1$. By Lemma \ref{vdC} we know that the image of $\mathcal{K}_G^0$ under $\phi_\beta$ is not a subset of $[0,1)$, when $a_0 > a_1$. Thus Lemma \ref{lem3} characterizes all van der Corput-type constructions when $d = 2$.
\end{remark}

\begin{theorem}\label{main2}
 Let $G^1, \ldots, G^s$ be numeration systems as in Theorem \ref{main1} and let $\beta_1, \ldots, \beta_s$ denote the roots of the corresponding characteristic equations. Then the $s$-dimensional, $\boldsymbol \beta$-adic Halton sequence $(\phi_{\boldsymbol \beta}(n))_{n \in \mathbb{N}}$ is u.d.\ in $[0,1)^s$.
\end{theorem}
\begin{proof}
  By Lemma \ref{lem3} and the definition of the Monna map we obtain an isometry between the dynamical systems $((\mathcal{K}_{G^1}, \tau_1) \times \ldots \times (\mathcal{K}_{G^s}, \tau_s))$ and $(([0,1), T_1) \times \ldots \times ([0,1), T_s))$ where
\begin{equation*}
 T_i \colon [0,1) \rightarrow [0,1), \quad T_i(x) := \phi_{\beta_i} \circ \tau_i \circ \phi_{\beta_i}^+ (x).
\end{equation*}
Let $\mathbf{T} \mathbf{x} = (T_1 x_1, \ldots, T_s x_s)$ for $\mathbf{x} = (x_1, \ldots, x_s) \in [0,1)^s$. Hence by Theorem \ref{B}, $(\mathbf{T}^n \mathbf{x})_{n \in \mathbb{N}}$ is u.d.\ in $[0,1)^s$ for all $\mathbf{x} \in [0,1)^s$. In particular $(\phi_{\boldsymbol \beta} (n))_{n \in \mathbb{N}} = (\mathbf{T}^n \mathbf{0})_{n \in \mathbb{N}}$ is u.d.
\end{proof}

\begin{remark}
 Note that the classical $b$-adic Halton sequence with pairwise coprime, integer bases $b_1, \ldots, b_s \geq 2$, is included in Theorem \ref{main2}.
\end{remark}

\begin{theorem}\label{101}
 Let the numeration system $G$ be defined by the coefficients $(a_0, a_1, a_2) = (1,0,1)$ and let $\beta$ be its characteristic root. Then $\mu(Z) = \lambda(\phi_{\beta}(Z))$ for all cylinder sets $Z$. Thus $T(x) = \phi_\beta \circ \tau \circ \phi_{\beta}^+ (x)$ is uniquely ergodic and $(T^n x)_{n \in \mathbb{N}}$ is u.d.\ for all $x$ in $[0,1)$. Furthermore the spectrum of $T$ is given by
\begin{equation}\label{spec}
  \Gamma = \left\{ \exp \left( 2 \pi i \frac{c}{\beta^l} \right) \colon m,l,c \in \mathbb{N} \cup \{0\} \right\}.
\end{equation}

\end{theorem}
\begin{proof}
Note that $\beta$ is a Pisot number and equation \eqref{par} holds since $\lfloor \beta \rfloor = 1 = a_0$. Hypothesis \ref{hypA} was proved for this case in \cite[Theorem 4]{bks}. The proof that Hypothesis \ref{hypB} is fulfilled can be found in \cite[Theorem 3]{aki}. Equation \eqref{spec} follows by the proof of Theorem \ref{main1}.\\

Now we have to prove that $\phi_\beta$ transports the measure $\mu$ to the Lebesgue measure on $[0,1)$. First we assume $k \geq 3$. Let the cylinder $Z$ be defined by the fixed digits $\epsilon_0, \ldots, \epsilon_{k - 1}$. We consider four different cases, first $\epsilon_{k-3} = \epsilon_{k-2} = \epsilon_{k-1} = 0$. Then $F_{k} = 1, F_{k+1} = 2, F_{k + 2} = 3$ and we get by \eqref{mu} that
\begin{equation*}
 \mu(Z) = \beta^{-k}.
\end{equation*}
Furthermore by the same argument as in the first part of the proof of Theorem \ref{main2} we obtain
\begin{equation*}
 \phi_\beta(Z) = \left[ \sum_{i = 0}^{k - 1} \frac{\epsilon_i}{\beta^{i+1}}, \sum_{i = 0}^{k - 2} \frac{\epsilon_i}{\beta^{i+1}} + \frac{(\epsilon_{k - 1} + 1)}{\beta^k} \right)
\end{equation*}
and thus $\lambda(\phi_\beta(Z)) = \beta^{-k}$.\\
Now let $\epsilon_{k-3} = 1, \epsilon_{k-2} = \epsilon_{k-1} = 0$. Hence $F_{k} = 1, F_{k+1} = 2, F_{k + 2} = 3$ and $\mu(Z) = \beta^{-k}$. We have
\begin{align*}
 \phi_\beta(Z) &= \left[\sum_{i = 0}^{k - 1} \frac{\epsilon_i}{\beta^{i+1}}, \sum_{i = 0}^{k - 1} \frac{\epsilon_i}{\beta^{i+1}} + \beta^{-k} \sum_{i = 0}^\infty \beta^{-(3 i+1)} \right)\\
  &= \left[\sum_{i = 0}^{k - 1} \frac{\epsilon_i}{\beta^{i+1}}, \sum_{i = 0}^{k - 1} \frac{\epsilon_i}{\beta^{i+1}} + \beta^{-k} \right)
\end{align*}
thus we have again $\lambda(\phi_\beta(Z)) = \beta^{-k}$. Now assume $\epsilon_{k-2} = 1, \epsilon_{k-1} = 0$. Hence $F_{k} = 1, F_{k+1} = 1, F_{k + 2} = 2$ and
\begin{equation*}
 \mu(Z) = \beta^{-k} \frac{\beta^{-2} + 1}{\beta^{-2} + \beta^{-1} + 1}.
\end{equation*}
Similarly as above we get
\begin{align*}
 \phi_\beta(Z) &= \left[ \sum_{i = 0}^{k - 1} \frac{\epsilon_i}{ \beta^{i+1}}, \sum_{i = 0}^{k - 1} \frac{\epsilon_i}{\beta^{i+1}} + \beta^{-(k+1)} \sum_{i = 0}^\infty \beta^{-(3 i+1)} \right)\\
  &= \left[\sum_{i = 0}^{k - 1} \frac{\epsilon_i}{\beta^{i+1}}, \sum_{i = 0}^{k - 1} \frac{\epsilon_i}{\beta^{i+1}} + \beta^{-(k+1)} \right)
\end{align*}
thus $\lambda(\phi_\beta(Z)) = \beta^{-(k+1)}$. Now we obtain
\begin{align*}
 \beta^{-(k+1)} &= \beta^{-k} \frac{\beta^{-2} + 1}{\beta^{-2} + \beta^{-1} + 1}\\
 \Leftrightarrow \beta^{-3} + \beta^{-2} + \beta^{-1} &= \beta^{-2} + 1
\end{align*}
which holds by \eqref{alpha}.\\
In the last case we assume $\epsilon_{k - 1} = 1$, thus  $F_{k} = F_{k+1} = F_{k + 2} = 1$ and
\begin{equation*}
 \mu(Z) = \beta^{-k} \frac{1}{\beta^{-2} + \beta^{-1} + 1}.
\end{equation*}
As above we get $\lambda(\phi_\beta(Z)) = \beta^{-(k+2)}$ and the result follows by
\begin{align*}
 \beta^{-(k+2)} &= \beta^{-k} \frac{1}{\beta^{-2} + \beta^{-1} + 1}\\
 \Leftrightarrow \beta^{-4} + \beta^{-3} + \beta^{-2} &= 1\\
 \Leftrightarrow \beta^{-3} + \beta^{-1} &= 1.
\end{align*}
The cases, where $k < 3$, follow by the same arguments.
\end{proof}

\begin{remark}
 As a consequence of Theorem \ref{product} we can construct two-dimensional u.d.\ sequences $(\phi_{\beta_1}(n), \phi_{\beta_2}(n))_{n \in \mathbb{N}}$, where $\beta_1$ is the characteristic root in Theorem \ref{101}, $(\phi_{\beta_2}(n))_{n \in \mathbb{N}}$ is the characteristic root of a numeration system in Theorem \ref{main2} and $\frac{\beta^k_1}{\beta^l_2} \notin \mathbb{Q}$ for all integers $k,l > 0$. In this way we can construct a new class of multidimensional u.d.\ sequences.\\
 Note that Theorem \ref{101} extends the examples given in \cite[Proposition 13,14]{bg}, where the authors consider $G$-additive functions which lead to u.d.\ point sequences in the unit interval.\\ 
 Furthermore, it is possible to show that the one-dimensional point sequence in the previous theorem is a low-discrepancy sequence by mimicking the proof for the $b$-adic van der Corput sequence, see e.g.\ \cite{kn, bg, carbone}.
\end{remark}

In \cite{carbone3}, the authors present the so-called Kakutani-Fibonacci-transformation for which they show that it is an ergodic transformation on the unit interval and the orbit of 0 is exactly the LS-sequence with parameters $L = S = 1$. With our approach we can show that this transformation is in fact uniquely ergodic, i.e.\ the orbit of $x$ under the transformation is u.d.\ for every $x \in [0,1)$.

\begin{theorem}\label{th3}
 The Kakutani-Fibonacci-transformation is uniquely ergodic.
\end{theorem}
\begin{proof}
 In \cite{ahz}, the authors use a van der Corput-type construction to compute the points of the LS-sequence with general parameters $L,S$. One can easily see that this construction is equivalent to the construction presented in the present paper, when $L = S = 1$, $G_{n + 2} = G_{n + 1} + G_n$, $G_0 = 1, G_1 = 2$ and $\beta$ is the golden ratio. Thus, by Theorem \ref{main2}, we get a uniquely ergodic transformation $T \colon [0,1) \rightarrow [0,1)$ given by $T(x) := \phi_\beta \circ \tau \circ \phi_\beta^+(x)$. Similar as for the classical van Neumann-Kakutani map, see \cite{ghl}, we observe that $T$ is exactly the piecewise translation map given in \cite{carbone3}. 
\end{proof}

\begin{remark}
 Note that, since $T$ is uniquely ergodic, the orbit of $x$ is u.d.\ for every $x \in [0,1)$. This fact is used in so-called randomized QMC techniques, where the starting point $X$ is a uniformly distributed random variable in $[0,1)$. This idea can of course be extended to the $s$-dimensional case. For more information on this topic and applications in financial mathematics see e.g.\ \cite{okt}.
\end{remark}

\bibliography{LSbib}
\bibliographystyle{abbrv}

\end{document}